\documentclass[reqno,12pt]{amsart}
\usepackage{amscd,amsfonts,mathrsfs,amsthm,amssymb, amsmath,mathtools}
\usepackage{upgreek}
\usepackage{csquotes}
\usepackage{enumerate}
\usepackage{bbm}
\usepackage{multicol}
\usepackage{stmaryrd,color}
\usepackage{array}
\usepackage{ae}
\usepackage{accents}
\usepackage{epsfig}
\usepackage[e]{esvect}
\usepackage[nobysame]{amsrefs}
\usepackage[margin=1in]{geometry}
\usepackage[toc,title,page]{appendix}
\usepackage{hyperref}
\hypersetup{colorlinks=true,linkcolor=blue,filecolor=blue,urlcolor=blue, citecolor=blue}

\let\oldtocsection=\tocsection
\let\oldtocsubsection=\tocsubsection
\let\oldtocsubsubsection=\tocsubsubsection
\renewcommand{\tocsection}[2]{\hspace{0em}\oldtocsection{#1}{#2}}
\renewcommand{\tocsubsection}[2]{\hspace{1em}\oldtocsubsection{#1}{#2}}
\renewcommand{\tocsubsubsection}[2]{\hspace{2em}\oldtocsubsubsection{#1}{#2}}

\def\equationcolor {\color{black}}
\def\textcolor     {\color{black}}

\def\bcoleq    {\begin{equation}\equationcolor}
\def\ecoleq    {\textcolor\end{equation}}
\def\bcoleqn   {\equationcolor\begin{eqnarray}}
\def\ecoleqn   {\end{eqnarray}\textcolor}

\def\S        {\mathbb{S}}

\def\scal{\operatorname{Scal}}

\def\det{\operatorname{det}}
\def\R{\mathbb{R}}
\def\TT{\mathbb{T}}

\DeclareMathOperator*{\Ric}{Ric}
\DeclareMathOperator*{\Scal}{Scal}

\DeclareMathOperator*{\Id}{Id}
\DeclareMathOperator*{\trace}{tr}

\newtheorem{theorem}{Theorem}[section]
\newtheorem{mythm}{Theorem}
\newtheorem{mycor}{Corollary}

\newtheorem{lemma}[theorem]{Lemma}

\theoremstyle{definition}
\newtheorem*{assumption*}{$\lambda_{1}$-Condition}

\newtheorem{example}[theorem]{Example}

\def\pproof#1{\@ifnextchar[\opargproof
{\opargproof[\it Proof of #1.]}}
\def\opargproof[#1]{\par\noindent {\bf #1 }}

\makeatother

\numberwithin{equation}{section}

\begin{document}

\title[Harmonic unit vector fields]{Harmonic unit vector fields on 3-manifolds}

\author[G. Habib]{\textsc{Georges Habib}}
\address{Georges Habib\newline
Lebanese University,
Faculty of Sciences II, Department of Mathematics, P.O. Box 90695
Fanar-Matn, Lebanon \& Universit{\'e} de Lorraine, CNRS, IECL, France,\newline
{\sl E-mail address:} {\bf ghabib@ul.edu.lb}
}

\author[A. Savas-Halilaj]{\textsc{Andreas Savas-Halilaj}}
\address{Andreas Savas-Halilaj\newline
University of Ioannina,
Section of Algebra \& Geometry,
45110 Ioannina, Greece,\newline
{\sl E-mail address:} {\bf ansavas@uoi.gr}
}

\renewcommand{\subjclassname}{  \textup{2000} Mathematics Subject Classification}
\thanks{This project was initiated when G. Habib was
a fellow at the Alfried Krupp Wissenschaftskolleg in Greifswald-Germany. He
would like to thank the institute for the support. A. Savas-Halilaj is supported in the framework of H.F.R.I. call ``Basic Research Financing” under the
National Recovery and Resilience Plan “Greece 2.0" funded by the European Union-NextGenerationEU
(HFRI Project Number: 14758).}
\parindent = 0 mm
\hfuzz     = 6 pt
\parskip   = 3 mm
\date{}

\begin{abstract}
We investigate harmonic unit vector fields with totally geodesic integral curves on $3$-manifolds. Under mild
curvature assumptions, we classify both the vector fields and the manifolds that support them. Our results are inspired by Carri\`ere’s classification of Riemannian flows on compact three-manifolds, as well as by the works of Geiges and Belgun on Killing vector fields on Sasakian manifolds.
\end{abstract}

\maketitle
\setcounter{tocdepth}{2}

\section{Introduction}

Suppose that $(M,g)$ is an $m$-dimensional manifold equipped with a Riemannian metric $g$.
A unit vector field $\zeta\in\mathfrak{X}(M)$ can be regarded as a ``graphical" map from $M$ to its unit tangent bundle
equipped with the Sasaki metric.
There is a natural functional that we may consider in the space of unit vector fields, i.e., the {\em energy
functional}
$$
E(\zeta)=\frac{m}{2}{\rm vol}(M,g)+\frac{1}{2}\int_M|\nabla\zeta|^2d\mu_g,
$$
whose critical points with respect to variations through nearby unit
vector fields are called {\em harmonic unit vector fields}; see, for example, \cite{wiegmink,wood}. The Euler-Lagrange equation for a critical
point of the energy functional is
\begin{equation}\label{harmonicunit}
\nabla^*\nabla\zeta-|\nabla\zeta|^2\zeta=0,
\end{equation}
where $\nabla^*\nabla$ is the trace (or rough) Laplacian
$$
\nabla^*\nabla=-\trace\nabla^{2}_{\boldsymbol{\cdot},\boldsymbol{\cdot}}.
$$
When we regard $\zeta$ as a map into the unit tangent bundle 
and allow arbitrary variations, then the Euler-Lagrange equation produces, 
besides \eqref{harmonicunit}, also
\begin{equation}\label{harmonicunitmap}
\trace R(\nabla_{\boldsymbol{\cdot}}\zeta, \zeta) \boldsymbol{\cdot} = 0,
\end{equation}
where $R$ is the Riemann curvature tensor of $M$; see, for instance, \cite{han}.  
In particular, a unit vector field on the unit sphere defines a harmonic 
map into the unit tangent bundle if it satisfies \eqref{harmonicunit}, is 
divergence-free, and has totally geodesic integral curves.
The classification of harmonic unit vector fields, and of the manifolds supporting them, is far from being understood. Even the case of harmonic unit vector fields on space forms remains wide open. For example, it is conjectured that a harmonic unit vector field on $\mathbb{S}^{2n+1}$ must be tangent to the fibers of the Hopf-fibration; see \cite{medrano-b,han,Ben1,fourtzis,perrone,brito2,perrone2,medrano3,medranoh,brito,Borrelli,Deturck,wood1}
for further details.

A unit vector field $\zeta$ on $M$ gives rise to a one-dimensional distribution $\mathcal{V}$, which we call the \emph{vertical distribution}. Its orthogonal complement is denoted by $\mathcal{H}$ and is called the \emph{horizontal distribution}. These two distributions play a crucial role in the analysis of the solutions of the partial differential system \eqref{harmonicunit}. 
There is a tensor that measures how much the horizontal distribution is twisted within the tangent bundle of $M$.
Namely, the $(1,1)$-tensor $\varphi$ defined by
\[
\varphi(X) \doteq -\nabla_X \zeta, \quad \text{for } X \in \mathfrak{X}(M),
\]
vanishes on the vertical distribution and behaves like the second fundamental form of a hypersurface. For this reason, $\varphi$ is 
called the (formal) \emph{second fundamental form} of $\mathcal{H}$. If $\zeta$ has totally geodesic integral curves, then
$\varphi$ is the {\em second fundamental form of the foliation.} In the case where $\varphi$ is anti-symmetric, the vector field is said to be \emph{Killing}, while in the case where it is anti-symmetric only on the horizontal distribution, it is called a \emph{Riemannian flow}.

Killing vector fields arise naturally in the setting of Sasakian manifolds.
According to the Uniformization Theorem \cite{Geiges, Belgun}, compact 3-dimensional Sasakian manifolds are completely understood.
Clearly, every Killing vector field is a Riemannian flow, but the converse is not necessarily true. Moreover, it is not necessarily true that a Killing unit vector field or a Riemannian flow is a harmonic unit vector field. 
It should be mentioned that according to a very beautiful theorem of Carri\`ere~\cite{carriere}, there exists a classification of Riemannian flows in 3-dimensions and of corresponding manifolds supporting them.

The purpose of this paper is to classify harmonic unit vector fields in three dimensions and the corresponding $3$-manifolds.  
We now state the main result.

\begin{mythm}\label{thmglobal}
Let $\zeta$ be a harmonic unit vector field, with totally geodesic fibers, on a compact Riemannian $3$-manifold $M$. Suppose that 
\begin{equation}\label{sasaki}
\Ric(\zeta)=\lambda\zeta,
\end{equation}
where $\lambda$ is a non-negative number. Then the norm of the second fundamental
form of the foliation is constant and $\zeta$ is divergence-free. Moreover, there exist only two possible cases for $M$ and $\zeta$:
\begin{enumerate}[\rm(1)]
\item Either $\zeta$ is Killing, and $M$ is diffeomorphic to one of the following:
$$
(\mathbb{R}\times N)/\Gamma,\quad \mathbb{S}^3/\Gamma,\quad \widetilde{SL}(2,\mathbb{R})/\Gamma\quad\text{or}
\quad {\rm Nil}^3/\Gamma,
$$
where $N$ is a $2$-dimensional, complete simply connected, Riemannian surface and $\Gamma$ is a discrete subgroup of the connected component of the corresponding isometry group, or;
\smallskip
\item $M$ has constant scalar curvature. Moreover, around each point of $M$, there are (local) orthonormal vector fields $\{\zeta_1=\zeta,\zeta_2,\zeta_3,\}$ such that
\begin{equation}\label{lie}
[\zeta_1,\zeta_2]=a_{123}\zeta_3,\quad[\zeta_1,\zeta_3]=a_{132}\zeta_2\quad\text{and}\quad[\zeta_2,\zeta_3]=a_{231}\zeta_1,
\end{equation}
where $a_{123}$, $a_{132}$ and $a_{231}$ are real constant completely determined in terms of $\lambda$, the scalar curvature and the norm
of the second fundamental form of the foliation; see \eqref{Riccifinal}. In particular, the universal covering of
$M$ is diffeomorphic to the unimodular Lie group whose Lie algebra is described by \eqref{lie}.
\end{enumerate}
The converse is also true; namely, a unit Killing vector field 
satisfying the curvature condition \eqref{sasaki},
where $\lambda$ is a non-negative constant, must be a harmonic unit vector field. Moreover, given a Lie group whose Lie bracket on its Lie algebra satisfies \eqref{lie} gives rise to a
harmonic unit vector field with totally geodesic integral curves.
\end{mythm}

The scalar curvature in the first family of Theorem \ref{thmglobal} is not necessarily constant.
Moreover, the Reeb vector field of any Sasakian 3-manifold with non-constant scalar curvature is a harmonic vector field, since the Ricci tensor satisfies the geometric condition \eqref{sasaki}
with $\lambda = 2$; see \cite{blair}.
It is a well-known fact that there are six connected, simply connected, three-dimensional unimodular Lie groups \cite{Milnor}.
Compact quotients of unimodular Lie groups by discrete subgroups are described in \cite{raymond}.
 
As an immediate consequence of Theorem \ref{thmglobal} and of the methods developed in its proof, we obtain the following results as corollaries:

\begin{mycor}\label{thmsphere}
Let $\zeta$ be a harmonic unit vector field with totally geodesic integral curves on $\S^3$.
Then $\zeta$ is the Hopf vector field.
\end{mycor}

\begin{mycor}\label{thmflat}
Let $\zeta$ be a harmonic unit vector field with totally geodesic integral curves on $\TT^3$.
Then, either $\zeta$ is parallel or a rotation along a parallel vector field.
\end{mycor}

\begin{mycor}\label{thmhyper}
There are no harmonic unit vector fields with totally geodesic integral curves on a compact
$3$-dimensional hyperbolic space.
\end{mycor}

\begin{mycor}\label{thmcontact}
Let $\zeta$ be a unit vector field with totally geodesic integral curves on a compact Riemannian $3$-manifold. If $\Ric(\zeta,\zeta)>0$, then
the dual form of the vector field $\zeta$ gives rise to a contact structure. 
\end{mycor}

Corollary \ref{thmsphere} was proved by Fourtzis, Markellos, and Savas-Halilaj \cite{fourtzis} while Corollary \ref{thmcontact} was shown by Gluck and Gu~\cite{gluck-gu} in the case of the 3-sphere $\mathbb{S}^3$.
In the case of the 3-dimensional hyperbolic space, and when $\zeta$ is divergence-free, Corollary \ref{thmhyper} was obtained by Perrone \cite{perrone}.

\section{Killing vector fields}
In this section, we recall some basic facts about Killing vector fields. In particular, we review the uniformization theorem for three-dimensional manifolds as stated in \cite{Geiges, Belgun}. We begin with the following lemma, which provides a condition under which a Killing vector field is harmonic.
\begin{lemma}\label{kill}
Let $(M,g)$ be a Riemannian manifold and let $\zeta$ be a unit Killing vector field.
If
$$\Ric(\zeta)=f \zeta,$$
for some smooth function $f$ on $M$, then
$f=|\varphi|^2\ge 0$
and $\zeta$ is a harmonic vector field. 
\end{lemma}
\begin{proof}
Since $\zeta$ is Killing, its integral curves are geodesics. The Hessian of $\zeta$ is given by Kostant's formula:
\begin{equation}\label{konstant}
\nabla^2_{X,Y}\zeta\doteq\nabla_{X}\nabla_Y\zeta-\nabla_{\nabla_{X}Y}\zeta = R(X,\zeta)Y, \quad \text{for all } X,Y \in \mathfrak{X}(M),
\end{equation}
see, for example, \cite[Proposition 8.1.3.]{petersen}.

\newpage
Taking a local orthonormal frame $(e_i)$ on $M$, and using the fact that $\varphi$ is anti-symmetric, we deduce
\[
f = \Ric(\zeta,\zeta) = \sum_i g(R(e_i,\zeta)\zeta,e_i)
   = \sum_i g(\nabla^2_{e_i,\zeta}\zeta,e_i)
   = \sum_i g(\nabla_{\varphi(e_i)}\zeta,e_i)
   = |\varphi|^2.
\]
Kostant's formula \eqref{konstant} can also be expressed in terms of $\varphi$ as
\[
(\nabla_X\varphi)Y = R(\zeta,X)Y, \quad \text{for all } X,Y \in \mathfrak{X}(M).
\]
Taking the trace yields
\[
-\nabla^*\nabla\zeta = -\Ric(\zeta) = -|\varphi|^2 \zeta,
\]
from which we conclude that $\zeta$ is harmonic.
\end{proof}

Let $(M,g)$ be a $(2m+1)$-dimensional Riemannian manifold, and let $\zeta$ be a unit vector field. 
We say that $(M,g,\zeta)$ is a \emph{Sasakian manifold} if $\zeta$ is a Killing vector field such that the
second fundamental form $\varphi$ of the foliation satisfies the following properties:
$$
\varphi^{2} = -{\rm Id} + \zeta \otimes \zeta\quad\text{and}\quad
 (\nabla_X \varphi)(Y) = \langle X, Y \rangle \zeta - \langle Y, \zeta \rangle X,\qquad\text{for all}\,\,\,
X, Y \in \mathfrak{X}(M).$$ 
These two conditions are equivalent to the statement that $\varphi$ defines a K\"ahler structure on the distribution $\mathcal{H}$.

In \cite{Geiges}, Geiges classified compact three-dimensional Sasakian manifolds. 
More precisely, he proved the following (see also \cite{Belgun}):

\begin{theorem}[Geiges, 1997]\label{thm:sasaki3manifold}
Let $(M,g,\zeta)$ be a compact three-dimensional Sasakian manifold. 
Then $M$ is diffeomorphic to one of the following:
\begin{enumerate}[\rm(1)]
  \item $\mathbb{S}^3 / \Gamma$, with $\Gamma \subset {\rm SO}(4)$,
  \medskip
  \item $\widetilde{SL}(2,\mathbb{R}) / \Gamma$, with $\Gamma \subset \mathfrak{I}_0(\widetilde{SL}(2,\mathbb{R}))$,
  \medskip
  \item ${\rm Nil}^3 / \Gamma$, with $\Gamma \subset \mathfrak{I}_0({\rm Nil}^3)$,
\end{enumerate}
where $\Gamma$ is a discrete subgroup of the connected component of the corresponding isometry group.
\end{theorem}

We would like to mention that $3$-dimensional Riemannian manifolds supporting Riemannian flows
were classified by Carri\`ere \cite{carriere}.

Let us now discuss some examples of harmonic unit vector fields, as well as examples of Riemannian flows.

\begin{example}[Hopf vector field]
Let us consider the unit Euclidean sphere $\mathbb{S}^{3}$ as a subset of $\mathbb{C}^{2}$, centered at the origin, and denote by $J$ its standard complex structure, i.e., left multiplication by $i \in \mathbb{C}$. The vector field
\[
\S^3\ni p\mapsto\zeta_p \doteq -J p\in T_{p}\S^3,
\]
is globally defined and tangent to the sphere.  
The vector field $\zeta$ is called the \emph{Hopf vector field}. Since it is a unit Killing vector field and
$$
\Ric(\zeta) = 2 \zeta,
$$
it follows from Lemma \ref{kill} that $\zeta$ is harmonic. This vector field induces a Riemannian flow.
\end{example}

\begin{example}[The hyperbolic torus]\label{carriere}
Let $A$ be a matrix in
\[
\mathrm{SL}(2,\mathbb{Z}) \doteq \{ A \in \mathrm{Mat}_{2 \times 2}(\mathbb{Z}) \mid \det(A) = 1 \},
\]
and consider the quotient manifold
$$\TT_A^3\doteq\TT^2\times\R/(m,t)\sim (A(m),t+1).$$
The space $\TT_A^3$ is called the {\em hyperbolic torus}.
The matrix $A$ has two eigenvalues $\beta>1$ and ${1}/{\beta}$ with corresponding eigenvectors $v_1$ and $v_2$. Equip $\TT_A^3$
with the Riemannian metric $g$ given by 
$$g\doteq\beta^{-2t} dx^2\oplus \beta^{2t} ds^2\oplus dt^2,$$
where $(x,s,t)$ are the local coordinates corresponding in the $v_2, v_1$, $\partial_t$ directions, respectively.

Consider the orthonormal frame
\[
e_1 \doteq \beta^{-t} \partial_s, \quad
e_2 \doteq \beta^t \partial_x, \quad
e_3 \doteq \partial_t.
\]
Then,
$$
[e_1,e_2]=0,\quad[e_1,e_3]=\ln(\beta)e_1,\quad [e_2,e_3]=-\ln(\beta)e_2,
$$
and from the Koszul formula we deduce that
$$
\nabla_{e_1}e_1=-\ln(\beta)e_3,\quad \nabla_{e_1}e_2=0,\quad \nabla_{e_1}e_3=\ln(\beta)e_1,
$$
and
$$
\nabla_{e_2}e_1=0,\quad \nabla_{e_2}e_2=\ln(\beta)e_3,\quad \nabla_{e_2}e_3=-\ln(\beta)e_2.
$$
Moreover,
$$
\nabla_{e_3}e_1=0,\quad \nabla_{e_3}e_2=0,\quad \nabla_{e_3}e_3=0.
$$
The Ricci tensor, with respect to the frame $(e_1, e_2, e_3)$, is
\[
\Ric = 
\begin{pmatrix}
0 & 0 & 0\\
0 & 0 & 0\\
0 & 0 & -2 (\ln(\beta))^2
\end{pmatrix}.
\]
The representation of the second fundamental form $\varphi_3\doteq-\nabla e_3$ with respect to the
basis $(e_1, e_2)$ is
\[
\varphi_3=
\begin{pmatrix}
-\ln(\beta) & 0\\
0 & \ln(\beta)
\end{pmatrix},
\]
which is symmetric rather than anti-symmetric. As a matter of fact, the vector field $e_3$ is a harmonic vector field, with totally geodesic integral curves, which is not Killing.
On the other hand, the unit vector field $e_1$ is harmonic, but its integral curves are not totally geodesic. In particular, the second fundamental form $\varphi_1\doteq-\nabla e_1$ is identically zero on the horizontal distribution
of $e_1$.
\end{example}

\begin{example}[Harmonic unit vector fields on the unimodular Lie group \cite{perrone, Milnor}]\label{Perrone} There exists a more general construction than the hyperbolic torus. 
Let $\mathfrak{g}$ be a simply connected $3$-dimensional Lie algebra generated by the vector fields $(e_1, e_2,e_3)$, with Lie brackets defined by 
\[
[e_1,e_2] = \alpha e_3, \quad [e_1,e_3] = \beta e_2, \quad [e_2, e_3] = \gamma e_1, 
\quad \text{for } \alpha, \beta, \gamma \in \mathbb{R}.
\]
On the associated Lie group $\mathscr{G}$, called the unimodular Lie group \cite{ Milnor}, consider the left-invariant metric such that 
$(e_1, e_2,e_3)$ is orthogonal at the identity. 
A straightforward computation shows that these vector fields have totally geodesic integral curves. Moreover,
the Levi-Civita connection satisfy
\[
\nabla_{e_3} e_1 = \frac{\gamma-\alpha - \beta}{2} e_2, \quad 
\nabla_{e_3} e_2 = \frac{\alpha + \beta - \gamma}{2} e_1, \quad 
\nabla_{e_1} e_3 = \frac{\beta-\alpha + \gamma}{2} e_2,
\]
and
\[
\nabla_{e_2} e_3 = \frac{\alpha + \beta + \gamma}{2} e_1, \quad 
\nabla_{e_2} e_1 = \frac{-\alpha - \beta - \gamma}{2} e_3, \quad 
\nabla_{e_1} e_2 = \frac{\alpha - \beta - \gamma}{2} e_3. 
\]
The matrix of the second fundamental form 
\[
\varphi = -\nabla e_3,
\]
of $\zeta$ with respect to the frame $(e_1, e_2)$ is given by
\[
\varphi = -\frac{1}{2}
\begin{pmatrix}
0 & \alpha + \beta + \gamma \\[4pt]
-\alpha + \beta + \gamma & 0
\end{pmatrix}.
\]
Moreover, the vector field $e_3$ is harmonic, since
\[
\nabla^*\nabla e_3=(\nabla_{e_1} \varphi)(e_1) + (\nabla_{e_2} \varphi)(e_2)
= \frac{(\alpha + \beta + \gamma)^2}{4} e_3 
+ \frac{(\alpha - \beta - \gamma)^2}{4} e_3
= |\varphi|^2 e_3.
\]
A similar computation shows that $e_1$ and $e_2$ are also harmonic unit vector fields with totally geodesic integral curves. 
The Ricci tensor, with respect to the basis $(e_1, e_2,e_3)$, is
\[
\Ric = -\frac{1}{2}
\begin{pmatrix}
(\alpha + \beta - \gamma)(\alpha + \beta + \gamma) & 0 & 0 \\[4pt]
0 & (\alpha - \beta - \gamma)(\alpha + \beta - \gamma) & 0 \\[4pt]
0 & 0 & (\alpha + \beta + \gamma)(-\alpha+\beta + \gamma)
\end{pmatrix}.
\]
In the special case $\alpha = -\beta$ and $\gamma = 0$,
the Ricci tensor vanishes, and hence the metric is flat. 
Therefore, the universal covering of $\mathscr{G}$ is diffeomorphic to $\mathbb{R}^3$. 
In this case, the vector field $e_1$ is parallel, and the distribution generated by $e_2$ and $e_3$ is integrable and totally geodesic. 
Therefore, the universal covering of $\mathscr{G}$ is diffeomorphic to $\mathbb{R}^3$.
If the metric is complete, it is isometric to the Euclidean one.
Thus, one can choose local coordinates $(x, y, z)$ such that $e_1 = \partial_x$.
From the computation of the connection forms, one easily deduces that $e_2$ and $e_3$ must take the form
\[
e_1 = \partial_x,\quad e_2 = \sin(\alpha x)\,\partial_y + \cos(\alpha x)\,\partial_z,
\quad
 e_3 = \cos(\alpha x)\,\partial_y - \sin(\alpha x)\,\partial_z,
\]
where $\alpha$ is a non-zero real number.
The vector fields $e_2$ and $e_3$ descend to $\mathbb{T}^3$ if $\alpha\in\mathbb{Z}$. 
 \end{example}

\section{General Bochner-Weitzenb\"ock formulas}\label{sec2}
In this section, we introduce the notation and review some fundamental results on harmonic unit vector fields. Although most of these computations hold in any dimension, we restrict our attention to the 3-dimensional case.

\subsection{Codazzi and Riccati type equations} Let $\zeta$ be a harmonic unit vector field with totally geodesic integral curves,
defined in an open
neighbourhood $V\subset M$.
It is well-known that the second fundamental form $\varphi$ of the foliation satisfies the equations
\begin{equation}\label{codazzi}
(\nabla_{\zeta}\varphi)(X)=\varphi^2(X)+R(X\,,\zeta)\zeta \quad\text{and}\quad(\nabla_{X}\varphi)(Y)
-(\nabla_{Y}\varphi)(X)=-R(X,Y)\zeta,
\end{equation}
for any $X,Y\in \mathfrak{X}(M)$
see, for example,
\cite[p. 313]{baird1}. On the other hand, harmonicity of the vector field $\zeta$ can be expressed equivalently 
in terms of the {\em divergence} $\delta\varphi$ of the second fundamental form $\varphi$. Namely, if $(e_i)$ is a local orthonormal
frame on $M$, then
\begin{equation}\label{dz}
-\delta\varphi\doteq\sum_{i}(\nabla_{e_i}\varphi)(e_i)=\sum_{i}\big(\nabla_{e_i}\varphi(e_i)-\varphi(\nabla_{e_i}e_i)\big)=|\varphi|^2\zeta.
\end{equation}
Let us make some comments about the quantities that appear
in \eqref{codazzi}.
\begin{itemize}
\item
The first equality in \eqref{codazzi} is a
Riccati and the second one is Codazzi type equation. One can easily see that the Riccati is a special case of Codazzi equation.
\smallskip
\item
Note that in general $\varphi$ is not symmetric neither
skew-symmetric. Let us denote the symmetric part of $\varphi$ by $S$.
We may write $S$ in the form
\begin{equation}\label{symmetric}
S\doteq\varphi+\varphi^T,
\end{equation}
where $\varphi^T$ is the transpose of $\varphi$.
Moreover, denote the skew-symmetric part of the second fundamental form $\varphi$ by $\tilde{S}$, that is
\begin{equation}\label{antisymmetric}
\tilde{S}\doteq\varphi-\varphi^T.
\end{equation}
\item
The $(1,1)$-tensor $L$ given by $L(\boldsymbol{\cdot})\doteq R(\boldsymbol{\cdot},\zeta)\zeta$,
is called the {\em Jacobi operator}.
\end{itemize}

In Lemmas \ref{codvarphitra} and \ref{laplacians} below, we derive several important identities involving the tensor $\varphi$.
These identities hold in all dimensions.

\begin{lemma}\label{codvarphitra}
The following facts hold true:
\begin{enumerate}[\rm(1)]
\item
 The $(1,1)$-tensor $\varphi^T$ satisfies
 \begin{equation}\label{cod111}
 \delta\varphi^T=\Ric(\zeta)-d(\trace (\varphi)) \quad\text{and}\quad
\nabla_{\zeta}\varphi^T=(\varphi^T)^2+L.
\end{equation}
\item
For any $X,Y,Z\in\mathfrak{X}(M)$, we have
$$\langle (\nabla_{X}\varphi^T)(Y)-(\nabla_{Y}\varphi^T)(X),Z \rangle=\langle(\nabla_Z \tilde{S})(X),Y\rangle-R(X,Y,\zeta,Z),$$
and, in particular,
\begin{equation}\label{eq:codS}
\langle (\nabla_{X}S)(Y)-(\nabla_{Y}S)(X),Z \rangle=\langle(\nabla_Z \tilde{S})(X),Y\rangle-2R(X,Y,\zeta,Z).
\end{equation}
\end{enumerate}
\end{lemma}
\begin{proof}
Consider a local orthonormal frame $(e_i)$ on $M$. To simplify the computations, consider vector fields $X,Y,Z$
that are parallel at some fixed point $x_0$ on $M$.

(1) Differentiating and estimating at $x_0$, and using \eqref{codazzi}, we get
\begin{eqnarray*}
-(\delta\varphi^T)(X)&\!\!\!=\!\!\!&\sum_{i}\langle(\nabla_{e_i}\varphi^T)(e_i),X\rangle\\
&\!\!\!=\!\!\!&\sum_{i}\langle (\nabla_{e_i}\varphi)(X),e_i\rangle
=\sum_{i}\langle(\nabla_X\varphi)(e_i)-R(e_i,X)\zeta,e_i\rangle.
\end{eqnarray*}
Therefore,
$$
-(\delta\varphi^T)(X)= X(\trace(\varphi))
-\Ric(\zeta,X),
$$
which concludes the proof of the first part. Moreover, again using \eqref{codazzi}, we deduce that
\begin{eqnarray*}
\langle(\nabla_{\zeta}\varphi^T)(X),Y\rangle
&\!\!\!=\!\!\!&\langle X,(\nabla_{\zeta}\varphi)(Y)\rangle
=\langle X,\varphi^2(Y)\rangle+R(Y,\zeta,\zeta,X)\\
&\!\!\!=\!\!\!&\langle (\varphi^T)^2X,Y\rangle+R(X,\zeta,\zeta,Y), 
\end{eqnarray*}    
which shows the required identity.

(2) With the help of \eqref{codazzi} and of the first Bianchi identity, we compute 
\begin{eqnarray*}
 \langle (\nabla_{X}\varphi^T)(Y)&\!\!\!-\!\!\!&(\nabla_{Y}\varphi^T)(X),Z \rangle = \langle Y,(\nabla_{X}\varphi)(Z)\rangle-\langle X,(\nabla_{Y}\varphi)(Z)\rangle\\
 &\!\!\!=\!\!\!&\langle Y,(\nabla_{Z}\varphi)(X)\rangle-R(X,Z,\zeta,Y)-\langle X,(\nabla_{Z}\varphi)(Y)\rangle+R(Y,Z,\zeta,X)\\
 &\!\!\!=\!\!\!&\langle Y,(\nabla_{Z}\varphi)(X)\rangle-\langle (\nabla_{Z}\varphi^T)(X),Y\rangle-R(X,Z,\zeta,Y)+R(Y,Z,\zeta,X)\\
 &\!\!\!=\!\!\!&\langle(\nabla_Z \tilde{S})(X),Y\rangle-R(X,Y,\zeta,Z).
\end{eqnarray*}
The Codazzi type equation for $S$ comes from those of $\varphi$ and $\varphi^T$.
\end{proof}

 Since the manifold is of dimension $3$, we denote by $J$ the complex structure on $\mathcal{H}$ and consider a local orthonormal frame of the form 
\[
\{e_1,\, e_2 \doteq J (e_1),\, \zeta\}.
\]
We extend $J$ to the tangent bundle $TM$ by setting
$$J(\zeta) \doteq 0.$$

\begin{lemma}\label{odes} 
The following relation holds for all $X, Y \in \mathfrak{X}(M)$:
\begin{equation}\label{eq:nablaJ}
(\nabla_X J)(Y)
= -\langle J\varphi(X), Y \rangle\, \zeta
  + \langle Y, \zeta \rangle\, J\varphi(X).
\end{equation}
Moreover, the functions $\trace(\varphi)$ and $\trace(\varphi J)$ satisfy the following differential equations along the leaves of the totally geodesic foliation:
\begin{equation}\label{ode2dimen3}
\begin{cases}
\zeta(\trace(\varphi J))\!\!\!&=\,\,\,\trace(\varphi)\trace(\varphi J),
\\
\zeta(\trace(\varphi))\!\!\!&=\,\,\,\trace(\varphi^2)+\Ric(\zeta,\zeta)=(\trace(\varphi))^2-2\det(\varphi)+\Ric(\zeta,\zeta),\\
\zeta(|\varphi|^2)\!\!\!&=\,\,\,2\trace(\varphi^2\varphi^T+\varphi^TL).
\end{cases}
\end{equation}
\end{lemma}
\begin{proof}
The identity \eqref{eq:nablaJ} follows by straightforward computations.
We begin by proving the first equation in \eqref{ode2dimen3}. 
From \eqref{eq:nablaJ} we observe that 
\[
\nabla_\zeta J = 0.
\]
Hence,
$$
\nabla_{\zeta}(\varphi J)=\varphi^2 +L J.
$$
Moreover, since $L$ is a symmetric tensor on $\mathcal{H}$, it follows that 
$$
\trace(L J)= 0.
$$
Keeping these facts in mind, and using \eqref{codazzi}, we obtain
\[
\zeta(\trace(\varphi J))
= \trace(\varphi^2 J)
= \trace(\varphi)\, \trace(\varphi J).
\]
To establish the second differential equation, first note that for any $2 \times 2$ matrix $\varphi$, we have
\[
\trace(\varphi^2) = \trace(\varphi)^2 - 2 \det(\varphi).
\] 
Using \eqref{codazzi}, we then have
\begin{eqnarray*}
\zeta(\trace(\varphi)) 
&\!\!\!=\!\!\!& \trace(\varphi^2 + L) 
= \trace(\varphi^2) + \Ric(\zeta, \zeta)\\
&\!\!\!=\!\!\!& \trace(\varphi)^2 - 2 \det(\varphi)+ \Ric(\zeta, \zeta).
\end{eqnarray*}
Finally, observe that
$$
\zeta(|\varphi|^2)=2\trace(\varphi^T\nabla_{\zeta}\varphi)=2\trace(\varphi^2\varphi^T+L \varphi^T).
$$
This completes the proof of the lemma.
\end{proof}

In the next lemma we compute the Laplacians of the second fundamental form and its trace.

\begin{lemma}\label{laplacians}
The following identities hold:
\begin{enumerate}[\rm(1)]
\item The Laplacian of $\trace(\varphi)$ is given by the formula
$$
\Delta\trace(\varphi)=|\varphi|^2\trace(\varphi)-2{\trace}(\varphi^2\varphi^T)
+\trace(\varphi^T(\Ric-2L))-\tfrac{1}{2}\zeta(\Scal).
$$
\item The Laplacian of the tensor $\varphi$ satisfies
\begin{eqnarray*}
(\nabla^*\nabla\varphi)(X)+\langle X,\nabla|\varphi|^2\rangle\zeta&\!\!\!=\!\!\!&|\varphi|^2\varphi(X)-\varphi(\Ric(X))
\\
&&-\sum_{i}(\nabla_{e_i}R)(X,e_i,\zeta)
+2\sum_{i}R(X,e_i)\varphi(e_i),
\end{eqnarray*}
for any $X\in\mathfrak{X}(M)$ and any local orthonormal frame $(e_i)$.
\end{enumerate}
\end{lemma}

\begin{proof} Consider a local orthonormal frame $(e_i)$ on $M$.
\begin{enumerate}[\rm(1)]
\item Observe first that the trace of $\varphi$ is simply the divergence of the vector field $\zeta$, i.e.,
\[
\trace(\varphi) = -\sum_i \langle \nabla_{e_i} \zeta, e_i \rangle = \delta \zeta.
\]
Taking the Laplacian $\Delta$ on both sides of the above equation, and using the commutativity of $\Delta$ with $\delta$,
together with the Weitzenb\"ock formula \cite[Theorem 9.4.1]{petersen} for the Hodge Laplacian and the third equation of \eqref{ode2dimen3}, we obtain
\begin{eqnarray}\label{f1}
\Delta \trace(\varphi)
&\!\!\!=\!\!\!& \Delta(\delta \zeta)
= \delta(\Delta \zeta)
= \delta(\nabla^* \nabla \zeta + \Ric(\zeta)) \nonumber\\
&\!\!\!=\!\!\!& \delta(|\varphi|^2 \zeta + \Ric(\zeta))
= |\varphi|^2 (\delta \zeta) - \zeta(|\varphi|^2) + \delta(\Ric(\zeta)) \nonumber\\
&\!\!\!=\!\!\!& |\varphi|^2 \trace(\varphi)
- 2 \trace(\varphi^2 \varphi^T)
- 2 \trace(\varphi^T L)
+ \delta(\Ric(\zeta)).
\end{eqnarray}

On the other hand,
\begin{eqnarray*}
2\,\delta(\Ric(\zeta))
&\!\!\!=\!\!\!& -2 \sum_i \langle \nabla_{e_i} \Ric(\zeta), e_i \rangle\\
&\!\!\!=\!\!\!& -2 \sum_i \langle (\nabla_{e_i} \Ric)(\zeta), e_i \rangle
+ 2 \sum_i \langle \varphi(e_i), \Ric(e_i) \rangle.
\end{eqnarray*}
Hence,
\begin{equation}\label{f2}
2\,\delta(\Ric(\zeta))
= 2 (\delta \Ric)(\zeta) + 2 \trace(\varphi^T \Ric)
= -\zeta(\Scal) + 2 \trace(\varphi^T \Ric).
\end{equation}
Combining \eqref{f1} with \eqref{f2}, we obtain the desired identity.
\smallskip
\item
Consider a vector field $X$ defined in a neighborhood of a fixed point $x_0 \in M$.
Without loss of generality, we may assume that at $x_0$ we have
$$\nabla X=\nabla e_i=0.$$
Differentiating \eqref{dz} with respect to $X$, using \eqref{codazzi}, and evaluating at $x_0$, we obtain
\begin{eqnarray*}
&&\hspace{-30pt}\nabla_X(|\varphi|^2\zeta)\\
&\!\!\!=\!\!\!&\sum_{i}\nabla_X\big(\nabla_{e_i}\varphi(e_i)-\varphi(\nabla_{e_i}e_i)\big)=\sum_{i}\big(\nabla_X\nabla_{e_i}\varphi(e_i)-\nabla_X\varphi(\nabla_{e_i}e_i)\big)\\
&\!\!\!=\!\!\!&\sum_{i}\big(R(X,e_i)\varphi(e_i)+\nabla_{e_i}\nabla_X\varphi(e_i)
-\varphi(\nabla_X\nabla_{e_i}e_i)\big)\\
&\!\!\!=\!\!\!&\sum_{i}R(X,e_i)\varphi(e_i)
+\sum_{i}\nabla_{e_i}\big((\nabla_X\varphi)(e_i)+\varphi(\nabla_Xe_i)\big)
-\sum_{i}\varphi(\nabla_X\nabla_{e_i}e_i)\\
&\!\!\!=\!\!\!&\sum_{i}R(X,e_i)\varphi(e_i)
+\sum_{i}\nabla_{e_i}\big((\nabla_{e_i}\varphi)(X)-R(X,e_i)\zeta\big)
+\sum_{i}\varphi(R(e_i,X,e_i))\\
&\!\!\!=\!\!\!&\sum_{i}R(X,e_i)\varphi(e_i)-\sum_{i}\nabla_{e_i}R(X,e_i)\zeta
-(\nabla^*\nabla\varphi)(X)
-\varphi(\Ric(X)).
\end{eqnarray*}
Consequently,
\begin{eqnarray*}
(\nabla^*\nabla\varphi)(X)&\!\!\!=\!\!\!&
-\sum_{i}(\nabla_{e_i}R)(X,e_i,\zeta)
+2\sum_{i}R(X,e_i)\varphi(e_i)\\
&&-\varphi(\Ric(X))-\langle X,\nabla|\varphi|^2\rangle\zeta+|\varphi|^2\varphi(X).
\end{eqnarray*}
\end{enumerate}
This completes the proof of lemma.
\end{proof}

In what follows, we shall restrict to the case where $\zeta$ is a harmonic unit vector field that is also an eigenvector of the Ricci tensor with constant eigenvalue, i.e.,
\[
\Ric(\zeta) = \lambda \zeta,
\]
for some $\lambda \in \mathbb{R}$.
The following lemma is of crucial importance for the proofs of our main results.
\begin{lemma}\label{laplacvarphij}
If $\Ric(\zeta)=\lambda \zeta$, for some $\lambda\in \mathbb{R}$, then $\trace(\varphi J)$ is a harmonic function, i.e.,
$$\Delta(\trace(\varphi J))=0.$$
In particular, if $M$ is compact, then $\trace(\varphi J)$ is constant.
\end{lemma}
\begin{proof}
The Laplacian of the function
$$\trace(\varphi J)=-\langle\varphi,J\rangle,$$
is given by the formula
$$\Delta(\langle\varphi,J\rangle)=\langle\nabla^*\nabla\varphi,J\rangle-2\langle\nabla\varphi,\nabla J\rangle+\langle\varphi,\nabla^*\nabla J\rangle.$$
Let us compute each term separately. 
Choose a local orthonormal frame $(e_i)$ and assume that it is parallel at some fixed point $x_0 \in M$, i.e.,
$$
\nabla e_i|_{x_0}=0.
$$ 
Suppose further that $X \in \mathfrak{X}(M)$ is a vector field defined in a neighborhood of $x_0$ and parallel at $x_0$. 
Differentiating, evaluating at $x_0$, and keeping in mind equation~\eqref{eq:nablaJ}, we obtain
$$
    (\nabla^*\nabla J)(X)=-\sum_{j}\nabla_{e_j}\big\{(\nabla_{e_j}J)(X)\big\}
    =\sum_{j}\nabla_{e_j}\big\{\langle J\varphi(e_j),X\rangle\zeta-\langle X,\zeta\rangle J\varphi(e_j)\big\},
$$    
from where we see that    
\begin{eqnarray*}
(\nabla^*\nabla J)(X)&\!\!\!=\!\!\!&-\big\{(\delta (J\varphi))(X)\big\}\zeta+\varphi\varphi^TJ(X)+J\varphi\varphi^T(X)+\langle X,\zeta\rangle \delta (J\varphi)\\
&\!\!\!=\!\!\!&-\big\{(\delta (J\varphi))(X)\big\}\zeta+|\varphi|^2J(X)+\langle X,\zeta\rangle \delta (J\varphi).
\end{eqnarray*}
In the last equation, we use the fact that
$$AJ+JA=\trace(A) J$$
for any symmetric matrix $A$.
We claim that $\delta (J\varphi)$ vanishes. Indeed, in view of \eqref{eq:nablaJ} and the facts
$$
\delta\varphi=-|\varphi|^2\zeta$$
and
$$J(\zeta)=0,$$
we have
$$-\delta(J\varphi)=\sum_j(\nabla_{e_j}J\varphi)(e_j)=\sum_j (\nabla_{e_j}J)(\varphi e_j)-J(\delta\varphi)=0.$$
Consequently,
\begin{equation}\label{laplacianJ}
\nabla^*\nabla J=|\varphi|^2J.
\end{equation}
Using the identity \eqref{eq:nablaJ}, we deduce
\begin{eqnarray*}
\langle\nabla\varphi,\nabla J\rangle&\!\!\!=\!\!\!&\sum_{j,k} \langle(\nabla_{e_j}\varphi)(e_k),(\nabla_{e_j} J)(e_k)\rangle\\
&\!\!\!=\!\!\!&-\sum_j\langle(\nabla_{e_j}\varphi)(J\varphi (e_j)),\zeta\rangle+\sum_{j}\langle(\nabla_{e_j}\varphi)(\zeta),J\varphi (e_j)\rangle\\
&\!\!\!=\!\!\!&-\sum_j\langle J\varphi (e_j),(\nabla_{e_j}\varphi^T)(\zeta)\rangle+\sum_{j}\langle(\nabla_{e_j}\varphi)(\zeta),J\varphi (e_j)\rangle.
\end{eqnarray*}
Because
$$\varphi(\zeta)=\varphi^T(\zeta)=0,$$
we see that
\begin{equation}\label{gradfj}
\langle\nabla\varphi,\nabla J\rangle=-\langle J\varphi,\varphi^T\varphi\rangle+\langle\varphi^2,J\varphi\rangle=|\varphi|^2\langle\varphi,J\rangle.
\end{equation}
To compute the terms involving $\nabla^*\nabla\varphi$, we use the expression in Lemma \ref{laplacians}(2).
Moreover, we will use the well-known identity
$$
(\delta R)(X,\zeta,Y)=\sum_i(\nabla_{e_i}R)(X,e_i,\zeta,Y)=(\nabla_\zeta \Ric)(X,Y)-(\nabla_Y \Ric)(\zeta, X),
$$
which is a consequence of the second Bianchi identity; see, for example, \cite[Exercise 3.4.8.]{petersen}.

We compute:  
\begin{eqnarray}\label{nonumber}
\langle\nabla^*\nabla\varphi,J\rangle&\!\!\!=\!\!\!&\sum_j\langle(\nabla^*\nabla\varphi)(e_j),J(e_j)\rangle\nonumber\\
&\!\!\!=\!\!\!&|\varphi|^2\sum_{j}\langle \varphi(e_j),J(e_j)\rangle-\sum_{j}\langle\varphi\Ric(e_j),J(e_j)\rangle\nonumber\\
&&-\sum_{i,j}\langle(\nabla_{e_i} R)(e_j,e_i,\zeta),J(e_j)\rangle+2\sum_{i,j}\langle R(e_j,e_i)\varphi(e_i),J(e_j)\rangle\nonumber\\
&\!\!\!=\!\!\!&|\varphi|^2\sum_{j}\langle \varphi(e_j),J(e_j)\rangle-\sum_{j}\langle\varphi\Ric(e_j),J(e_j)\rangle\nonumber\\
&&-\sum_j(\delta R)(e_j,\zeta,J(e_j))+2\sum_{i,j}\langle R(e_j,e_i)\varphi(e_i),J(e_j)\rangle\nonumber\\
&\!\!\!=\!\!\!&|\varphi|^2\sum_{j}\langle \varphi(e_j),J(e_j)\rangle-\sum_{j}\langle\varphi\Ric(e_j),J(e_j)\rangle\nonumber\\
&&-\sum_j(\nabla_\zeta\Ric)(Je_j,e_j)+\sum_j(\nabla_{J(e_j)}\Ric)(\zeta,e_j)\nonumber\\
&&+2\sum_{i,j}\langle R(e_j,e_i)\varphi(e_i),J(e_j)\rangle.
\end{eqnarray}
Because the Ricci tensor is symmetric, $J$ is antisymmetric on $\mathcal{H}$, vanishes on $\mathcal{V}$, and $\nabla_\zeta J = 0$, a direct computation shows that the third sum in \eqref{nonumber} vanishes. By differentiating the equation $\Ric \, \zeta = \lambda \zeta$, the fourth sum in \eqref{nonumber} is equal to $\lambda \langle \varphi, J \rangle + \langle \Ric, \varphi J \rangle$. Finally, using the fact
$$\langle R(e_1,e_2)e_2,e_1\rangle=\frac{\Ric(e_1,e_1)+\Ric(e_2,e_2)-\Ric(\zeta,\zeta)}{2}=\frac{\scal-2\lambda}{2},$$
the last sum can be easily shown to be equal to 
$$\langle\varphi,J\rangle R(e_1,e_2,e_2,e_1)=\langle\varphi,J\rangle\frac{\Scal-2\lambda}{2}.$$
Since
$$(\varphi-\varphi^T)J=\trace(\varphi J)\Id|_{\mathcal{H}}=-\langle \varphi,J\rangle \Id|_{\mathcal{H}},$$
we arrive at the conclusion
\begin{eqnarray}\label{laplacianf}
\langle\nabla^*\nabla\varphi,J\rangle&\!\!\!=\!\!\!&|\varphi|^2\langle \varphi,J\rangle-\langle\varphi\Ric,J\rangle+\lambda\langle \varphi,J\rangle+\langle \Ric,\varphi J\rangle+\langle \varphi, J\rangle(\Scal-2\lambda)\nonumber\\
&\!\!\!=\!\!\!&|\varphi|^2\langle \varphi,J\rangle+\langle \Ric,(\varphi-\varphi^T) J\rangle+\langle \varphi,J\rangle(\Scal-\lambda)\nonumber\\
&\!\!\!=\!\!\!&|\varphi|^2\langle \varphi,J\rangle.
\end{eqnarray}
Combining the equations \eqref{laplacianJ}, \eqref{gradfj} and \eqref{laplacianf}, we deduce that
$\trace(\varphi J)$ is a harmonic function.
\end{proof}

\begin{lemma}\label{36}
If $\Ric(\zeta)=\lambda \zeta$ for some $\lambda\in \mathbb{R}$, then we have 
$$\Delta \trace(\varphi) = \trace(\varphi)\left(2\det\varphi-|\varphi|^2+\Scal-3\lambda\right)-\zeta(\Scal).$$
In particular, if $M$ is compact and $\Scal\leq 3\lambda$, we get that $\zeta$ must be divergence-free. 
\end{lemma}
\begin{proof}
Recall by \eqref{f1} that
$$\Delta\trace(\varphi)=|\varphi|^2\trace(\varphi)-2\trace{(\varphi^2\varphi^T)}-2\trace(\varphi^TL)+\delta(\Ric(\zeta)).$$
Consider a local orthonormal frame $(e_1,e_2,\zeta)$ on $M$ that consists of eigenvectors of $S$, i.e.,
$$S(e_1)=\beta_1 e_1,\quad S(e_2)=\beta_2e_2\quad\text{and}\quad S(\zeta)=0.$$
Let us compute the last three terms of the last identity separately. We have
\begin{eqnarray}\label{eq:tracevarphiS}
2\trace(\varphi^TL)&\!\!\!=\!\!\!&\langle L,S\rangle
=\big(\beta_1\langle L(e_1),e_1\rangle+\beta_2\langle L(e_2),e_2\rangle\big)\nonumber\\
&\!\!\!=\!\!\!&\big(\beta_1\langle R(e_1,\zeta)\zeta,e_1\rangle+\beta_2\langle R(e_2,\zeta)\zeta,e_2\rangle\big)\nonumber\\
&\!\!\!=\!\!\!&\big(\beta_1\Ric(e_1,e_1)-\beta_1\langle R(e_1,e_2)e_2,e_1\rangle+\beta_2\Ric(e_2,e_2)-\beta_2\langle R(e_2,e_1)e_1,e_2\rangle\big)\nonumber\\
&\!\!\!=\!\!\!&\big(\langle \Ric,S\rangle-\trace(S)\langle R(e_1,e_2)e_2,e_1\rangle\big)\nonumber\\
&\!\!\!=\!\!\!&2\langle \Ric,\varphi\rangle-\trace(\varphi)\big(\Scal-2\lambda\big),
\end{eqnarray}
where in the last equality, we use the identity
$$\langle R(e_1,e_2)e_2,e_1\rangle=\frac{\Ric(e_1,e_1)+\Ric(e_2,e_2)-\Ric(\zeta,\zeta)}{2}
=\frac{\scal-2\lambda}{2}.$$
Replacing \eqref{eq:tracevarphiS} into the expression of the Laplacian, using equation \eqref{f2} and the facts $$\Ric(\zeta) = \lambda \zeta\quad\text{and}\quad \delta \zeta = \trace(\varphi),$$
we see that
\begin{eqnarray*}
\Delta \trace(\varphi) &=& |\varphi|^2 \trace(\varphi) - 2 \trace(\varphi^2 \varphi^T) - 2 \trace(\varphi^T \Ric) + \trace(\varphi)(\Scal - 2\lambda) + \delta(\Ric(\zeta)) \\
&=& |\varphi|^2 \trace(\varphi) - 2 \trace(\varphi^2 \varphi^T) - \zeta(\Scal) + \trace(\varphi)(\Scal - 2\lambda) - \delta(\Ric(\zeta)) \\
&=& |\varphi|^2 \trace(\varphi) - 2 \trace(\varphi^2 \varphi^T) - \zeta(\Scal) + \trace(\varphi)(\Scal - 3\lambda).
\end{eqnarray*}
Finally, using the algebraic identity
\[
|\varphi|^2 \trace(\varphi) - 2 \trace(\varphi^2 \varphi^T) = \trace(\varphi) \big( 2 \det \varphi - |\varphi|^2 \big),
\]
we derive the required equation. To show the last part, we proceed as in \cite{fourtzis}. Since
\[
2 \det \varphi - |\varphi|^2 \leq 0,
\]
we obtain
\[
\Delta (\trace(\varphi)^2) = 2 \trace(\varphi)^2 \big( 2 \det \varphi - |\varphi|^2 + \Scal - 3\lambda \big) - 2 |\nabla (\trace(\varphi))|^2 \leq 0.
\]
Hence, as $M$ is compact, $\trace(\varphi)$ must be constant. But
\[
\trace(\varphi) = \delta \zeta,
\]
so its integral over $M$ is zero by Stokes' theorem. Therefore, this constant must be zero. This completes the proof.
\end{proof}

Next, we are going to prove the key point of Theorem \ref{thmglobal}.
\begin{lemma}\label{prop:threedimecase}
Let $\zeta$ be a divergence-free harmonic unit vector field with totally geodesic integral curves on $M$.  
If $\Ric(\zeta) = \lambda \zeta$ for some $\lambda \in \mathbb{R}$, then $|\varphi|$ is constant. Moreover, by letting 
\[
b^2 = \frac{|\varphi|^2 + \lambda}{4} \quad \text{and} \quad \lambda_1^2 = |\varphi|^2 - \lambda,
\]
either $\zeta$ is  Killing (in this case $\zeta$ is parallel, or up to a conformal change of the metric depending on $b$, the manifold $M$ is Sasakian) or, 
we have the following cases:
\begin{enumerate}[\rm(1)]
\item Either $b = 0$ and hence $\Scal = \lambda = \text{\rm constant}$, or;
\smallskip
\item $b \neq 0$, in which case $\Scal$ must be constant and, around each point, there exists a (local) orthonormal frame $(\zeta_1 = \zeta, \zeta_2, \zeta_3)$ such that
\[
\begin{array}{rcl}
[\zeta_1, \zeta_2] &=& \Big(-\dfrac{\Scal - \lambda}{4b} + \dfrac{\lambda_1}{2} - b\Big) \zeta_3, \\[0.8em]
[\zeta_1, \zeta_3] &=& \Big(\dfrac{\Scal - \lambda}{4b} + \dfrac{\lambda_1}{2} + b\Big) \zeta_2, \\[0.8em]
[\zeta_2, \zeta_3] &=& -2b\,\zeta_1.
\end{array}
\]
In this case, the Ricci tensor in the basis $(\zeta_1 = \zeta, \zeta_2, \zeta_3)$ is given by
\begin{equation}\label{Riccifinal}
\Ric =
\begin{pmatrix}
\lambda & 0 & 0 \\
0 & \dfrac{\Scal - \lambda}{2} \Big(1 + \dfrac{\lambda_1}{2b}\Big) & 0 \\
0 & 0 & \dfrac{\Scal - \lambda}{2} \Big(1 - \dfrac{\lambda_1}{2b}\Big)
\end{pmatrix}.
\end{equation}
\end{enumerate}
\end{lemma}
\begin{proof} 
We first show that $|\varphi|$ is constant in the $\zeta$-direction. Indeed, since $\varphi$ has vanishing trace, we use \eqref{f1} and the third identity of \eqref{ode2dimen3} to obtain
\[
\zeta(|\varphi|^2) = \delta(\Ric(\zeta)) = \lambda \, \delta\zeta = 0.
\]
To show that the second fundamental form $\varphi$ has constant norm in the other directions, we need to perform some computations. For simplicity, we set
\[
(d^\nabla S)(X,Y) \doteq (\nabla_X S)(Y) - (\nabla_Y S)(X), \quad \text{for all } X,Y \in \mathfrak{X}(M).
\]
Let $(e_i)$ be a local orthonormal frame on $M$. First observe that
\begin{equation}\label{eq:normsphi}
|S|^2 = 2|\varphi|^2 + 2\trace(\varphi^2)\quad\text{and}\quad
|\tilde{S}|^2=2|\varphi|^2-2\trace(\varphi^2).
\end{equation}
On the other hand, from \eqref{ode2dimen3}, we have that
\begin{equation}\label{normsphisss}
\trace(\varphi^2)=-\Ric(\zeta,\zeta)=-\lambda={\rm constant}.
\end{equation}
From the equation \eqref{eq:codS}, we deduce that
\begin{eqnarray}\label{eq:dnables}
\sum_i \langle (d^\nabla S)(e_i, \tilde{S}(e_i)), X \rangle
&\!\!\!=\!\!\!& \sum_i \langle (\nabla_X \tilde{S})(e_i), \tilde{S}(e_i) \rangle - 2 \sum_i R(e_i, \tilde{S}(e_i), \zeta, X) \nonumber\\
&\!\!\!=\!\!\!& \frac{1}{2} X(|\tilde{S}|^2) - 2 \sum_i R(e_i, \tilde{S}(e_i), \zeta, X) \nonumber\\
&\!\!\!=\!\!\!& \frac{1}{2} X(|S|^2) - 2 \sum_i R(e_i, \tilde{S}(e_i), \zeta, X).
\end{eqnarray}
Now, in view of \eqref{dz}, Lemma \ref{codvarphitra}(1), \eqref{eq:normsphi} and \eqref{normsphisss} we have that  
\begin{equation} \label{eq:deltaS}
\delta S=\delta(\varphi+\varphi^T)=(-|\varphi|^2+\lambda)\zeta=-(|S|^2/2)\zeta.
\end{equation}
Consider a local orthonormal frame field of the form $(e_1,e_2,\zeta)$ on $M$. Keeping in mind \eqref{eq:deltaS} and the symmetry of $S$, we compute 
\begin{eqnarray*}
 \langle(d^\nabla S)(e_1,e_2),e_1\rangle&\!\!\!=\!\!\!&\langle (\nabla_{e_1}S)(e_2)-(\nabla_{e_2}S)(e_1), e_1\rangle\\
 &\!\!\!=\!\!\!&\langle (\nabla_{e_1}S)(e_1),e_2\rangle-\langle(\nabla_{e_2}S)(e_1), e_1\rangle\\
 &\!\!\!=\!\!\!&\langle-(\nabla_{e_2}S)(e_2)+(|S|^2/2)\zeta,e_2\rangle-\langle(\nabla_{e_2}S)(e_1), e_1\rangle\\
  &\!\!\!=\!\!\!&\langle-(\nabla_{e_2}S)(e_2),e_2\rangle-\langle(\nabla_{e_2}S)(e_1), e_1\rangle\\
&\!\!\!=\!\!\!&-e_2(\trace(S))=-2e_2(\trace(\varphi))\\
 &\!\!\!=\!\!\!&0.
\end{eqnarray*}
In the same way, we prove that
$$\langle(d^\nabla S)(e_1,e_2),e_2\rangle=0.$$
Letting $X=e_1$ in \eqref{eq:dnables}, we get that 
$$
 \frac{1}{2}e_1(|S|^2)=\sum_{i}\langle (d^\nabla S)(e_i,\tilde{S}(e_i)),e_1 \rangle  +2\sum_{i} R(e_i,\tilde{S}(e_i),\zeta,e_1),
$$
from where we deduce that
\begin{eqnarray*} 
  \frac{1}{2}e_1(|S|^2)&=&2\langle\tilde{S}(e_1),e_2\rangle \langle (d^\nabla S)(e_1,e_2),e_1 \rangle+4  \langle\tilde{S}(e_1),e_2\rangle R(e_1,e_2,\zeta,e_1)\\
 &=&4 \langle\tilde{S}(e_1),e_2\rangle \Ric(\zeta,e_2)\\
 &=&0.
\end{eqnarray*}
Following the same steps, we arrive at the conclusion
$e_2(|S|^2)=0.$
From Equation \eqref{eq:normsphi}, it follows that the norm of the second fundamental form $\varphi$ is constant.

Using Lemma \ref{36} and since $\varphi$ is divergence-free we deduce that
\begin{equation}\label{zscal}
\zeta(\scal)=0.
\end{equation}
From \eqref{f2} it follows that
\begin{equation}\label{ricS}
0=\trace(\varphi^T\Ric)=\langle \Ric,\varphi\rangle=\frac{\langle\Ric,S\rangle}{2}.
\end{equation}
Since $S$ is symmetric, trace-free, with constant norm, there exists a local orthonormal frame $(e_1,e_2)\in\mathcal{H}$ such that
$S(e_1)=\lambda_1e_1$ and $S(e_2)=-\lambda_1e_2,$
where $\lambda_1$ is constant. In this case, from the second equation of \eqref{ode2dimen3}, we deduce that
$$\det\varphi=\frac{\lambda}{2}.$$
As a consequence from \eqref{eq:normsphi}, the matrix of $\varphi$ with respect to the frame $(e_1,e_2)$ can be written in the form
\begin{equation}\label{matrixphi}
\varphi=
\begin{pmatrix}
\frac{\lambda_1}{2}&b\\
-b&-\frac{\lambda_1}{2}
\end{pmatrix}
\end{equation}
where
\begin{equation}\label{beq}
b^2=\frac{\lambda}{2}+\frac{\lambda_1^2}{4}={\rm constant}.
\end{equation}
From \eqref{ricS} it follows that
\begin{equation}
0=\lambda_1(\Ric(e_1,e_1)-\Ric(e_2,e_2)).
\end{equation}
Since $\lambda_1$ is constant, only two cases occur:\\\\
{\bf Case 1:} Suppose $\lambda_1=0$. In this case $S=0$ and $\zeta$ is a unit Killing vector field.
Here, two sub-cases arise:

\textbf{Sub-Case 1.} If $\lambda = 0$, the tensor $\varphi$ vanishes identically. 
In this case, $\zeta$ is a parallel vector field.

\textbf{Sub-Case 2.} If $\lambda \neq 0$, then, from \eqref{beq}, it follows that $\lambda$ is positive. 
Therefore, by performing a conformal change of the metric and $\zeta$ of the form 
\[
\tilde{g} \doteq b^{2} g \quad\text{and}\quad \tilde{\zeta} \doteq \frac{1}{b}\,\zeta,
\]
the manifold $(M, \tilde{g}, \tilde{\zeta})$ becomes Sasakian. 
Namely, the tensor $\tilde{\varphi}$ is given by
\[
\tilde{\varphi} = -\widetilde{\nabla}\tilde{\zeta} = -\frac{1}{b}\nabla\zeta 
= 
\begin{pmatrix}
0 & -1 \\[4pt]
1 & 0
\end{pmatrix},
\]
and defines a K\"ahler structure on $\mathcal{H}$.

{\bf Case 2:} Suppose $\lambda_1\neq 0$. Then,
$$
\Ric(e_1,e_1)=\Ric(e_2,e_2)=\frac{\scal-\lambda}{2}.
$$
In this situation, we have
\begin{equation}\label{eq:sectionc}
R(e_1,e_2,e_2,e_1)=\frac{\Ric(e_1,e_1)+\Ric(e_2,e_2)-\Ric(\zeta,\zeta)}{2}=\frac{\scal-2\lambda}{2}.
\end{equation}
From Lemma \ref{codvarphitra}(1) it follows on $\mathcal{H}$ that
\begin{equation}\label{nousias}
\nabla_{\zeta}S=\nabla_{\zeta}\varphi+\nabla_{\zeta}\varphi^T=\varphi^2+(\varphi^T)^2+2L=
-\lambda {\rm Id}|_{\mathcal{H}}+2R(\cdot,\zeta)\zeta.
\end{equation}
Note that
$$
(\nabla_{\zeta}S)(e_1)=\nabla_{\zeta}(Se_1)-S(\nabla_{\zeta}e_1)=\lambda_1\nabla_{\zeta}e_1+
\lambda_1\langle\nabla_{\zeta}e_1,e_2\rangle e_2=2\langle\nabla_{\zeta}e_1,e_2\rangle \lambda_1e_2.
$$
Evaluating \eqref{nousias} at $e_1$, and comparing with  the last equation, we get
\begin{equation}\label{ric12}
\Ric(e_1,e_2)=\lambda_1\langle\nabla_{\zeta}e_1,e_2\rangle.
\end{equation}
Using \eqref{eq:deltaS}, we deduce that
\begin{eqnarray*}
\lambda_1^2\zeta&=&(|S|^2/2)\zeta=(\nabla_{e_1}S)(e_1)+(\nabla_{e_2}S)(e_2)\\
&=& \lambda_1\nabla_{e_1}e_1-S(\nabla_{e_1}e_1)-\lambda_1\nabla_{e_2}e_2-S(\nabla_{e_2}e_2)\\
&=&\lambda_1\nabla_{e_1}e_1+\lambda_1\langle \nabla_{e_1}e_1,e_2\rangle e_2-
\lambda_1\nabla_{e_2}e_2-\lambda_1\langle \nabla_{e_2}e_2,e_1\rangle e_1\\
&=& \lambda_1 \langle\nabla_{e_1}e_1,\zeta\rangle\zeta+2\lambda_1\langle \nabla_{e_1}e_1,e_2\rangle e_2
-\lambda_1\langle \nabla_{e_2}e_2,\zeta\rangle\zeta-2\lambda_1\langle \nabla_{e_2}e_2,e_1\rangle e_1\\
&=&\lambda_1^2\zeta+2\lambda_1\langle \nabla_{e_1}e_1,e_2\rangle e_2-2\lambda_1\langle \nabla_{e_2}e_2,e_1\rangle e_1,
\end{eqnarray*}
from where it follows that
\begin{equation}\label{connforms123}
\langle \nabla_{e_1}e_1,e_2\rangle=0
\quad
\text{and}
\quad
\langle \nabla_{e_2}e_2,e_1\rangle=0.
\end{equation}
Consequently, from \eqref{connforms123} and \eqref{matrixphi}, it follows that
\begin{equation}\label{neweqns}
\nabla_{e_1}e_1=\frac{\lambda_1}{2}\zeta,\,\,\,\nabla_{e_2}e_2=-\frac{\lambda_1}{2}\zeta,\,\,\,
\nabla_{e_1}e_2=-b\zeta,\,\,\,\nabla_{e_2}e_1=b\zeta.
\end{equation}
Differentiating the first equation of \eqref{connforms123} the direction of $\zeta$, and
using \eqref{ric12}, \eqref{neweqns} and \eqref{matrixphi}, we get
\begin{eqnarray*}
0&\!\!\!=\!\!\!&\langle\nabla_{\zeta}\nabla_{e_1}e_1,e_2\rangle+\langle\nabla_{e_1}e_1,\nabla_{\zeta}e_2\rangle\\
&\!\!\!=\!\!\!&\langle R(\zeta,e_1)e_1,e_2\rangle+\langle \nabla_{e_1}\nabla_{_\zeta}e_1,e_2\rangle
+\langle\nabla_{[\zeta,e_1]}e_1,e_2\rangle+\langle \nabla_{e_1}e_1,\zeta\rangle\langle\nabla_{\zeta}e_2,\zeta\rangle\\
&\!\!\!=\!\!\!&\Ric(\zeta,e_2)+\langle \nabla_{e_1}\nabla_{_\zeta}e_1,e_2\rangle+\langle\nabla_{\nabla_{\zeta}e_1}e_1,e_2\rangle-\langle\nabla_{\nabla_{e_1}\zeta}e_1,e_2\rangle\\
&\!\!\!=\!\!\!&\lambda\langle \zeta,e_2\rangle+\langle \nabla_{e_1}\nabla_{_\zeta}e_1,e_2\rangle+\langle\nabla_{\nabla_{\zeta}e_1}e_1,e_2\rangle-\langle\nabla_{\nabla_{e_1}\zeta}e_1,e_2\rangle\\
&\!\!\!=\!\!\!&\langle \nabla_{e_1}\nabla_{_\zeta}e_1,e_2\rangle=e_1\langle\nabla_{\zeta}e_1,e_2\rangle
-\langle\nabla_{\zeta}e_1,\nabla_{\zeta}e_2\rangle\\
&\!\!\!=\!\!\!&(1/\lambda_1)e_1(\Ric(e_1,e_2)).
\end{eqnarray*}
Following the same lines with the second equation of \eqref{connforms123}, we show that the derivative of
$\Ric(e_1,e_2)$ in the $e_2$-direction is also zero. Therefore,
\begin{equation}\label{riiccc}
e_1(\Ric(e_1,e_2))=0=e_2(\Ric(e_1,e_2)).
\end{equation}
Differentiating the first equation of \eqref{connforms123} the $e_2$-direction, using, \eqref{beq}, \eqref{eq:sectionc} and the formulas for the
connection forms, we obtain
\begin{eqnarray*}
0&\!\!\!=\!\!\!&\langle\nabla_{e_2}\nabla_{e_1}e_1,e_2\rangle+\langle\nabla_{e_1}e_1,\nabla_{e_2}e_2\rangle\\
&\!\!\!=\!\!\!&\langle R(e_2,e_1)e_1,e_2\rangle+\langle\nabla_{e_1}\nabla_{e_2}e_1,e_2\rangle
+\langle\nabla_{[e_2,e_1]}e_1,e_2\rangle-\lambda^{2}_1/4\\
&\!\!\!=\!\!\!&(\scal-2\lambda)/2-\langle\nabla_{e_2}e_1,\nabla_{e_1}e_2\rangle+2b\langle\nabla_{\zeta} e_1,e_2\rangle-\lambda^{2}_1/4\\
&\!\!\!=\!\!\!&(\scal-2\lambda)/2+b^2+2b\langle\nabla_{\zeta} e_1,e_2\rangle-\lambda^{2}_1/4\\
&\!\!\!=\!\!\!&(\scal-2\lambda)/2+\lambda/2+2b\langle\nabla_{\zeta}e_1,e_2\rangle.
\end{eqnarray*}
In the above computation, we use that $[e_2,e_1]=2b\zeta$ from \eqref{neweqns}.
Hence,
\begin{equation}\label{scalb}
2b\langle\nabla_{\zeta}e_1,e_2\rangle=-\frac{\scal-\lambda}{2}.
\end{equation}

{\bf Sub-Case 1:} If $b=0$, then we deduce from \eqref{scalb} that $\scal=\lambda={\rm constant}$.

{\bf Sub-Case 2:} If $b\neq 0$, we obtain from \eqref{ric12} that
$$
\Ric(e_1,e_2)=-\lambda_1\frac{\scal-\lambda}{4b}
$$
which, with the help of \eqref{riiccc}, yields
$$
e_1(\scal)=0=e_2(\scal).
$$
Combining with \eqref{zscal}, it follows that $\scal$ is constant. Moreover,
\begin{equation}\label{christoffel1}
\nabla_{e_1}e_1=\frac{\lambda_1}{2}\zeta,\quad \nabla_{e_2}e_2=-\frac{\lambda_1}{2}\zeta,\quad
\nabla_{e_1}e_2=-b\zeta,\quad \nabla_{e_2}e_1=b\zeta,
\end{equation}
and
\begin{equation}\label{christoffel2}
\nabla_{\zeta}e_1=-\left(\frac{\scal-\lambda}{4b}\right)e_2,\quad \nabla_{\zeta}e_2=\left(\frac{\scal-\lambda}{4b}\right)e_1.
\end{equation}
Furthermore,
\begin{equation}\label{christoffel3}
\nabla_{e_1}\zeta=-\frac{\lambda_1}{2}e_1+be_2\quad\text{and}\quad \nabla_{e_2}\zeta=-be_1+\frac{\lambda_1}{2}e_2.
\end{equation}
Consider now the local orthonormal frame
$$\Big(\zeta_1\doteq\zeta,\zeta_2\doteq\frac{e_1-e_2}{\sqrt{2}}, \zeta_3\doteq\frac{e_1+e_2}{\sqrt{2}}\Big).$$
Keeping in mind \eqref{christoffel1}, \eqref{christoffel1} and \eqref{christoffel1}, direct computations yield
$$
[\zeta,\zeta_2]=\left(-\frac{\scal-\lambda}{4b}+\frac{\lambda_1}{2}-b\right)\zeta_3,\,\,\,\,\,\, [\zeta,\zeta_3]=\left(\frac{\scal-\lambda}{4b}+\frac{\lambda_1}{2}+b\right)\zeta_2,\,\,\,\,\,\,
[\zeta_2,\zeta_3]=-2b\zeta.
$$
According to Example \ref{Perrone}, the Ricci tensor with respect to this frame is given by
\begin{equation*}
\Ric=
\begin{pmatrix}
\lambda&0&0\\
0&\frac{\scal-\lambda}{2}\big(1+\frac{\lambda_1}{2b}\big)&0\\
0&0&\frac{\scal-\lambda}{2}\big(1-\frac{\lambda_1}{2b}\big) 
\end{pmatrix}.
\end{equation*}
This completes the proof.
\end{proof} 

\section{Proofs of the main results}
We are now ready to prove the main results of the paper.

{\bf Proof of Theorem \ref{thmglobal}:} At first we show that  $\zeta$ is divergence-free and the norm of $\varphi$ is constant.
Since $M$ is compact, we deduce from Lemma \ref{laplacvarphij} that
$$\trace(\varphi J)=c\in\R.$$
By the first equation in \eqref{ode2dimen3} we deduce that
$$c\trace(\varphi)=0.$$
If $c$ is not zero then we immediately obtain the result. Assume now that $c=0$.
In this case $\varphi=\varphi^T$ and the second equation in \eqref{ode2dimen3} reduces to
$$\zeta(\trace(\varphi))=|\varphi|^2+\lambda.$$
At a minimum point $x_0$  of $\trace(\varphi)$, we have that
$$0\le |\varphi|^2(x_0)=-\lambda\le 0.$$
Consequently $\lambda=0$. Observe now that at the minimum and maximum points $x_0$ and $x_1$ of $\trace(\varphi)$, we have
$$\varphi|_{x_0}=\varphi|_{x_1}=0.$$
Therefore, we deduce that
$$0=\trace(\varphi|_{x_0})\leq \trace(\varphi)\leq \trace(\varphi|_{x_1})=0.$$
Hence $\trace(\varphi)=0$ at each point.  The last statement follows from Lemma \ref{prop:threedimecase}.

\begin{enumerate}[\rm(1)]
\item If $\zeta$ is Killing, then either $\zeta$ is parallel or, up to a conformal change of the metric, 
the manifold $M$ is Sasakian. 
Hence, either the universal cover of $M$ is a direct product $\mathbb{R} \times N$, 
where $N$ is a complete and simply connected surface (so that 
$M$ is diffeomorphic to a quotient $\mathbb{R} \times N / \Gamma$) or, 
by the Uniformization Theorem \ref{thm:sasaki3manifold}, 
the manifold $M$ is diffeomorphic to one of the following:
\[
\mathbb{S}^3 / \Gamma, \qquad 
\widetilde{SL}(2,\mathbb{R}) / \Gamma, \qquad 
{\rm Nil}^3 / \Gamma,
\]
where $\Gamma$ is a discrete subgroup of the connected component of the corresponding isometry group.

\smallskip
\item Assume now that $\zeta$ is not Killing. 
Since $\Ric(\zeta) = \lambda \zeta$ with $\lambda \ge 0$, 
the function $b$ in Lemma~\ref{prop:threedimecase} cannot vanish, 
as this would imply $\lambda_1 = 0$ (contradicting $\lambda \ge 0$). 
The universal covering of $M$ is therefore diffeomorphic to a unimodular Lie group \cite{kobayachi, Milnor,palais}. 
This completes the proof of the theorem.
\end{enumerate}

The converse is also true; namely, a unit Killing vector field 
satisfying the curvature condition \eqref{sasaki},
where $\lambda$ is a non-negative constant, must be a harmonic unit vector field as shown in Lemma \ref{kill}. Also, it is shown in Example  \ref{Perrone} that given a Lie group whose Lie bracket on its Lie algebra satisfies \eqref{lie} gives rise to a unit totally geodesic harmonic vector field.

This completes the proof.
\hfill$\square$

{\bf Proof of Corollary \ref{thmsphere}:} In the case of $\S^3$ the only possibility is when $\lambda_1=0$, since otherwise the Ricci tensor cannot take the form of \eqref{Riccifinal}. Then $\zeta$ should be Killing. On the other hand is it well-known that the Hopf vector field is the only Killing unit vector field.\hfill$\square$

{\bf Proof of Corollary \ref{thmflat}:} In this case $\lambda=0$. Then either $\zeta$ is Killing on $\TT^3$ and
 $\zeta$ is parallel,
or we are in the second case. By using the equations on the Christoffel symbols we find that for
$$b=-\lambda_1/2$$
we have
$$
[\zeta,\zeta_2]=\lambda_1\zeta_3,\quad [\zeta,\zeta_3]=0,\quad [\zeta_2,\zeta_3]=\lambda_1\zeta.
$$
Then we use the computation in Example \ref{Perrone}  to conclude. \hfill$\square$

{\bf Proof of Corollary \ref{thmhyper}:} By compactness, Lemma \ref{36}, and integration it follows that $\trace(\varphi)=0$.
As in the proof of Corollary \ref{thmsphere}, the Ricci tensor in Lemma \ref{prop:threedimecase} cannot take the form of  \eqref{Riccifinal}. Hence, it follows that $\zeta$ should be Killing. But it is known that Killing vector fields cannot occur on compact
manifolds with negative Ricci tensor; see, for example, \cite{petersen}. \hfill$\square$

{\bf Proof of Corollary \ref{thmcontact}:}
First, we show that $\trace(\varphi J)$ is nowhere vanishing. Indeed, assume that  $\trace(\varphi J)$ is zero at some point $x_0$, then from the first equation of \eqref{ode2dimen3}, it vanishes along the integral curve $\gamma$ of $\zeta$ passing through $x_0$. Therefore, along $\gamma$, the tensor $\varphi$ has the form
$$\varphi=\begin{pmatrix}a&b\\b&d\end{pmatrix}.$$
Observe that along $\gamma$ the equation \eqref{ode2dimen3} takes the form
\begin{equation}\label{integral}
\zeta(\trace(\varphi))=\trace(\varphi^2)+\Ric(\zeta,\zeta)=a^2+d^2+2b^2+\Ric(\zeta,\zeta)>\varepsilon.
\end{equation}
Because $M$ is compact, the integral curves of $\zeta$ are complete, i.e., they are defined for all values in $\R$.
On the other hand, the function $\trace(\varphi)$
is bounded from above. Integrating \eqref{integral} we deduce that
$$
\trace(\varphi)(\gamma(s))\ge\varepsilon s+\trace(\varphi)(\gamma(0)),
$$
which leads to a contradiction.

To show that the dual form
of the vector field $\zeta$ defines a contact structure, we take $\{\zeta,e_1,e_2\}$ a local frame and compute 
\begin{eqnarray*}
  (\zeta\wedge d\zeta) (\zeta,e_1,e_2)&\!\!\!=\!\!\!&d\zeta(e_1,e_2)\\
  &\!\!\!=\!\!\!&-(e_1\wedge\varphi(e_1)+e_2\wedge\varphi(e_2))(e_1,e_2)\\
  &\!\!\!=\!\!\!&\langle\varphi(e_2),e_1\rangle-\langle\varphi(e_1),e_2\rangle\\
  &\!\!\!=\!\!\!&\trace(\varphi J), 
\end{eqnarray*}
which is nowhere zero as shown before. This finishes the proof. \hfill$\square$


\begin{bibdiv}
\begin{biblist}

\bib{baird1}{book}{
   author={Baird, P.},
   author={Wood, J.},
   title={Harmonic morphisms between Riemannian manifolds},
   series={London Mathematical Society Monographs. New Series},
   volume={29},
   publisher={The Clarendon Press, Oxford University Press, Oxford},
   date={2003},
}

\bib{Belgun}{article}{
   author={Belgun, F.},
   title={On the metric structure of non-K\"ahler complex surfaces},
   journal={Math. Ann},
   volume={317},
   date={2000},
   pages={1-40},
}





\bib{Ben1}{article}{
 author={Benyounes, M.},
 author={Loubeau, E.},
 author={Wood, C.M.},
 title={Harmonic vector fields on space forms},
 journal={Geom. Dedicata},
 volume={177},
 pages={323-352},
 date={2015},
}





\bib{blair}{book}{
   author={Blair, D.E.},
   title={Riemannian geometry of contact and symplectic manifolds},
   series={Progress in Mathematics},
   volume={203},
   publisher={Birkh\"{a}user Boston, Ltd., Boston, MA},
   date={2010},
}


\bib{Borrelli}{article}{
 author={Borrelli, V.},
 author={Brito, F.},
 author={Gil-Medrano, O.},
 title={The infimum of the energy of unit vector fields on odd-dimensional spheres},
 journal={Ann. Global Anal. Geom.},
 volume={23},
 number={2},
 pages={129-140},
 date={2003},
}

\bib{brito2}{article}{
 author={Brito, F.},
 author={Gomes, A.},
 author={Nunes, G.},
 title={Energy and volume of vector fields on spherical domains},
 journal={Pac. J. Math.},
 volume={257},
 pages={1-7},
 date={2012},
}

\bib{brito}{article}{
   author={Brito, F.},
   title={Total bending of flows with mean curvature correction},
   journal={Diff. Geom. and its Appl.},
   volume={334},
   date={2000},
   pages={157-163},
}

\bib{carriere}{article}{
   author={Carri{\`e}re, Y.},
   title={Flots riemanniens},
   journal={Structure transverse des feuilletages, Toulouse 1982, Ast{\'e}risque},
   volume={116},
   date={1984},
   pages={31-52},
}


\bib{Deturck}{article}{
 author={DeTurck, D.},
 author={Gluck, H.},
 author={Storm, P},
 title={Lipschitz minimality of Hopf fibrations and Hopf vector fields},
 journal={Algebr. Geom. Topol.},
 volume={13},
 pages={1369-1412},
 date={2013},
}

\bib{fourtzis}{article}{
   author={Fourtzis, I.},
   author={Markellos, M.},
   author={Savas-Halilaj, A.},
   title={Gauss maps of harmonic and minimal great circle fibrations},
   journal={Ann. Global Anal. Geom.},
   volume={63},
   date={2023},
   pages={Article No 12},
}

\bib{Geiges}{article}{
   author={Geiges, H.},
   title={Normal contact structures on $3$-manifolds},
   journal={Tohoku Math. J.},
   volume={49},
   date={1997},
   pages={415-422},
}




\bib{medrano-b}{book}{
 author={Gil-Medrano, O.},
 book={
 title={The volume of vector fields on Riemannian manifolds. Main results and open problems},
 publisher={Cham: Springer},
 },
 title={The volume of vector fields on Riemannian manifolds. Main results and open problems},
 series={Lecture Notes in Mathematics},
 volume={2336},
 date={2023},
 publisher={Springer, Cham},
}

\bib{medranoh}{article}{
   author={Gil-Medrano, O.},
   author={Hurtado, A.},
   title={Volume, energy and generalized energy of unit vector fields on Berger spheres: stability of Hopf vector fields},
   journal={Proc. R. Soc. Edinb., Sect. A, Math.},
   volume={135},
   date={2005},
   pages={789-813},
}

\bib{medrano3}{article}{
   author={Gil-Medrano, O.},
   author={Llinares-Fuster, E.},
   title={Second variation of volume and energy of vector fields. Stability of Hopf vector fields},
   journal={Math. Ann.},
   volume={320},
   date={2001},
   pages={531-545},
}

\bib{gluck-gu}{article}{
   author={Gluck, H.},
   author={Gu, W.},
   title={Volume-preserving great circle flows on the $3$-sphere},
   journal={Geom. Dedicata},
   volume={88},
   date={2001},
   pages={259-282},
}


\bib{gluck3}{article}{
   author={Gluck, H.},
   author={Warner, F.},
   title={Great circle fibrations of the three-sphere},
   journal={Duke Math. J.},
   volume={50},
   date={1983},
   pages={107-132},
}


\bib{han}{article}{
   author={Han, D.-S.},
   author={Yim, J-W.},
   title={Unit vector fields on spheres, which are harmonic maps}
   journal={Math. Z.},
   volume={227},
   date={1998}
   pages={83-92},
}

\bib{kobayachi}{book}{
   author={Kobayashi, S.},
   title={Transformation groups in differential geometry},
   series={Ergebnisse der Mathematik und ihrer Grenzgebiete},
   volume={70},
   publisher={Springer-Verlag},
   date={1972},
}





 
\bib{Milnor}{article}{
   author={Milnor, J.},
  title={Curvatures of left invariant metrics on Lie groups}
   journal={Adv. Math.},
   volume={21},
   date={1976}
   pages={293-329},
}

\bib{palais}{book}{
 author={Palais, R.},
 title={A global formulation of the Lie theory of transformation groups},
 series={Mem. Am. Math. Soc.},
 volume={22},
 pages={123},
 date={1957},
}

\bib{perrone2}{article}{
 author={Perrone, D.},
 title={On the energy of a unit vector field},
 journal={Ann. Global Anal. Geom.},
 volume={28},
 pages={91-106},
 date={2005},
}

\bib{perrone}{article}{
   author={Perrone, D.},
   title={Unit vector fields on real space forms which are harmonic maps},
   journal={Pac. J. Math.},
   volume={239},
   date={2009},
   pages={89-104},
  }



\bib{petersen}{book}{
 author={Petersen, P.},
 book={
 title={Riemannian geometry},
 publisher={Cham: Springer},
 },
 title={Riemannian geometry},
 edition={3rd edition},
 series={Graduate Texts in Mathematics},
 volume={171},
 date={2016},
}

\bib{raymond}{article}{
   author={Raymond, F.}
   author= {Vasquez, A.},
   title={Closed $3$-manifolds whose universal covering is a Lie
group},
   journal={Topology and its Appl.},
   volume={12},
   date={1981},
   pages={161-179},
}

  
 \bib{wiegmink}{article}{
   author={Wiegmink, G.},
   title={Total bending of vector fields on Riemannian manifolds},
   journal={Math. Ann.},
   volume={303},
   date={1995},
   pages={325-344},
  }

\bib{wood}{article}{
   author={Wood, C.M.},
   title={On the energy of a unit vector field},
   journal={Geom. Dedicata},
   volume={64},
   date={1997},
   pages={319-330},
}

\bib{wood1}{article}{
 author={Wood, C.M.},
 title={The energy of Hopf vector fields},
 journal={Manuscripta Math.},
 volume={101},
 pages={71-88},
 date={2000},
}
		
\end{biblist}
\end{bibdiv}
\end{document}